\documentclass[12pt]{amsart}
\usepackage[english]{babel}
\usepackage{tikz-cd}
\usepackage{cite}
\usepackage{url}
\usepackage{hyperref}
\usepackage{anysize}
\usepackage{graphicx}
\usepackage{fullpage}
\usepackage{colonequals}
\usepackage{amssymb}
\usepackage[leqno]{amsmath}
\usepackage{amscd}
\usepackage{amsthm}
\usepackage{amsmath}
\usepackage{enumitem}
\usepackage{xcolor}
\usepackage{marvosym}
\usepackage{soul} 
\usepackage{amsaddr}
\usepackage{amsthm} 

\marginsize{3cm}{3cm}{1cm}{2cm}
\def\R{\mathbb{R}}
\def\Q{\mathbb{Q}}
\def\A{\mathbb{A}}
\def\Z{\mathbb{Z}}
\def\C{\mathbb{C}}
\def\O{\mathcal{O}}
\def\H{\mathcal{H}}
\def\T{\mathcal{T}}

\def\PGL{\mathrm{PGL}}
\def\PSL{\mathrm{PSL}}
\def\TT{\boldsymbol{T}}
\def\GG{\boldsymbol{G}}
\def\sG{\widetilde{\boldsymbol{G}}}
\def\KK{\boldsymbol{K}}

\def\a{\mathfrak{a}}
\def\b{\mathfrak{b}}

\def\m{\mathfrak{m}}
\def\p{\mathfrak{p}}
\def\q{\mathfrak{q}}

\def\A{\mathbb{A}}
\def\P{\mathbb{P}}

\def\d{\partial}
\usepackage{amsmath}
\usepackage[all,cmtip]{xy}
\newtheorem{theorem}{Theorem}[section]
\newtheorem*{theorem*}{Lema}

\newtheorem{definition}{Definition}
\newtheorem{lemma}[theorem]{Lemma}
\newtheorem{proposition}[theorem]{Proposition}
\newtheorem{corollary}[theorem]{Corollary}
\theoremstyle{remark}
\newtheorem{remark}{Remark}

\numberwithin{equation}{section}

\newenvironment{customtheorem}[1]{%
  \renewcommand\thetheorem{#1}
  \theorem
}{\endtheorem}

\newcommand\rquot[2]{
  \mathchoice
  {
    \text{\raise0.5ex\hbox{$#1$}\big/\lower0.5ex\hbox{$#2$}}%
  }
  {
    #1\,/\,#2
  }
  {
    #1\,/\,#2
  }
  {
    #1\,/\,#2
  }
}

\newcommand\lrquot[3]{
  \mathchoice
  {
    \text{\lower0.5ex\hbox{$#1$}\big\backslash\raise0.5ex\hbox{$#2$\!}\big/
      \lower0.5ex\hbox{\!\!$#3$}}%
  }
  {
    #1\,\backslash\,#2\,/\,#3
  }
  {
    #1\,\backslash\,#2\,/\,#3
  }
  {
    #1\,\backslash\,#2\,/\,#3
  }
}

\newcommand\lquot[2]{
  \mathchoice
  {
    \text{\lower0.5ex\hbox{$#1$}\big\backslash\raise0.5ex\hbox{$#2$}}%
  }
  {
    #1\,\backslash\,#2
  }
  {
    #1\,\backslash\,#2
  }
  {
    #1\,\backslash\,#2
  }
}

\DeclareFontFamily{U}{wncy}{}
\DeclareFontShape{U}{wncy}{m}{n}{<->wncyr10}{}
\DeclareSymbolFont{mcy}{U}{wncy}{m}{n}
\DeclareMathSymbol{\Sh}{\mathord}{mcy}{"58}
\DeclareFontFamily{U} {cmmi}{}

\DeclareFontShape{U}{cmmi}{m}{n}{
  <-6> cmmi5
  <6-7> cmmi6
  <7-8> cmmi7
  <8-9> cmmi8
  <9-10> cmmi9
  <10-12> cmmi10
  <12-> cmmi12}{}

\DeclareSymbolFont{Xcmmi} {U} {cmmi}{m}{n}

\DeclareMathSymbol{\xi}{\mathord}{Xcmmi}{24}

\title{$p$-adic equidistribution of modular geodesics and of Heegner points on Shimura curves}
\author{Patricio Pérez-Piña}
\address{Pontificia Universidad Católica de Chile}
\email{perezpinha.patricio@gmail.com}
\date{\today}

\begin{document}
\begin{abstract}
We propose a $p$-adic version of Duke's Theorem on the equidistribution of closed geodesics on modular curves. Our approach concerns quadratic fields split at $p$ as well as a $p$-adic covering of the modular curve. We also prove an equidistribution result of Heegner points in the $p$-adic space attached to Shimura curves. 
\end{abstract}

\maketitle

\vspace{0.5cm}

\section{Introduction}

In \cite{D}, W. Duke proves the equidistribution, inside the complex points of the modular curve endowed with the hyperbolic measure, of arithmetic objects attached to orders in quadratic extensions of $\Q$. The nature of these objects depends on the behavior at infinity of the extension. For quadratic imaginary fields these objects are Galois orbits of CM points and in the case of real quadratic fields his equidistribution theorem concerns packets of closed geodesics. 

Let $p$ be a rational prime. Since CM points are defined over number fields, their Galois orbits can be naturally embedded inside the $p$-adic space of $\C_p$-points of the modular curve. In \cite{HMRL} and \cite{HMRL}, the authors studied the distribution of such Galois orbits obtaining $p$-adic analogues to Duke's theorem in the case of imaginary quadratic fields. In contrast with Duke's result, different limit measures can appear, depending on the discriminants  of the orders involved in the sequence of CM points. 

The aim of this article is to state and prove a $p$-adic version of Duke's theorem on the equidistribution of packets of closed geodesics. Also, we prove a $p$-adic version of the equidistribution of Galois orbits of Heegner points on Shimura curves, when the prime $p$ divides the discriminant of the underlying quaternion algebra. The complex version of the latter result is another extension of Duke's theorem for which we refer to section 6 in \cite{Michel2} or to \cite{Zhang}.

We recall Duke's result following section 2 of \cite{ELMV2}. Let $\H$ be the complex upper half plane and denote by $Y_0$ the modular curve of level $1$. Then we have an identification between the quotient space $\PSL_2(\Z)\backslash\H$ and $Y_0(\C)$. Moreover, the quotient space $\PSL_2(\Z)\backslash\PGL_2(\R)^+$ is identified with $T^1(Y_0(\C))$, the unit tangent bundle of $Y_0(\C)$. Here and elsewhere, the superscript $^+$ indicates that we consider matrices of positive determinant. Suppose that $\tau$ is a real quadratic irrational. We can view $\tau$ as an element in the boundary of $\H$ and we can consider the geodesic $\gamma_\tau$ in $\H$ connecting $\tau$ and its Galois conjugate. The projection of $\gamma_\tau$ to ${\PSL_2(\Z)}\backslash{\H}$ is a closed geodesic. We lift this closed geodesic to a compact geodesic orbit $\widetilde{\gamma}_\tau$ in $T^1(Y_0(\C))$ that depends only on the $\PSL_2(\Z)$-orbit of $\tau$. Also, we denote by $\O_\tau$ the ring of matrices in $M_2(\Z)$ having the column $(\tau\,\,\, 1)^t$ as an eigenvector. For $K$ a real quadratic field, we denote by $RQ(K)$ the set of $\tau$ as before such that $\O_\tau$ is isomorphic to the maximal order $\O_K$. The set $\PSL_2(\Z)\backslash RQ(K)$ is finite and its cardinality equals the class number $h_K$ of $K$. We put $\mathcal{G}_K=\bigcup_\tau \widetilde{\gamma}_\tau$, with $\tau$ running through representatives of $\PSL_2(\Z)\backslash RQ(K)$. Define by $\mu_K$ the unique probability measure in $T^1(Y_0(\C))$ supported on $\mathcal{G}_K$ that is invariant under the geodesic flow. Endow $T^1(Y_0(\C))$ with the unique probability measure $\mu_L$ coming from a Haar measure on $\PGL_2^+(\R)$ and a counting measure on $\PSL_2(\Z)$.

\begin{theorem}[Duke] As $disc(K)\to\infty$, the sequence $\mu_K$ converges to $\mu_L$ in the weak-* topology. In other words, the collection $\mathcal{G}_K$ of $h_K$ closed geodesics becomes equidistributed in $T^1(Y_0(\C))$ with respect to the measure $\mu_L$.
\end{theorem}

We proceed to explain our main result on geodesics. As we have seen, in the complex case, geodesics appear when dealing with real quadratic fields or in other words when the infinite place is split in the extension. Following this point of view, we exchange the role of $\infty$ and $p$ and we focus on quadratic extensions such that $p$ splits. The role played by $\PSL_2(\Z)$ and $\PGL_2^+(\R)$ will be held by $\Gamma^+\colonequals\PGL^+_2(\Z[1/p])$ and $\PGL_2(\R\times\Q_p)$, respectively. Let $\H'$ be the set of points $\tau$ in $\H$ such that $\tau$ generates an imaginary quadratic extension over $\Q$ in which $p$ splits. The quotient $\lquot{\Gamma^+}{\H'}$ was studied by Ihara who showed that it is related to the special fiber of the modular curve $X_0(p)$. 

\begin{theorem}[Chapter 5, Theorem 1 in \cite{Ihara}] Fix $\mathfrak{P}$, a prime above $p$ in $\overline{\Z}$, the integral closure of $\Z$ in $\overline{\Q}$. Identify $\overline{\Z}/\mathfrak{P}$ with $\overline{\mathbb{F}_p}$ and recall that by the Theory of Complex Multiplication, for any $\tau\in\H'$, we have $j(\tau)\in\overline{\Z}$. The assignment \[\Gamma^+\tau\mapsto \mathrm{Gal}(\overline{\mathbb{F}_p}/\mathbb{F}_p)\cdot(j(\tau)\mod\mathfrak{P})\] defines a bijection between $\Gamma^+\backslash\H'$ and Galois orbits in $\overline{\mathbb{F}_p}-SS$, where $SS$ is the set of supersingular $j$-invariants.
\end{theorem}

The $p$-adic covering of $Y_0(\C)$ that we consider is the locally compact space \[T^{(p)}(Y_0(\C))\colonequals\lquot{\Gamma^+}{(\H\times \PGL_2(\Q_p))}.\] The natural projection $T^{(p)}(Y_0(\C))\to Y_0(\C)$ is given by taking a right quotient by $\PGL_2(\Z_p)$ (see section \ref{IScycles}). Let $\T_p$ denote the Bruhat-Tits tree of $\PGL_2(\Q_p)$. After identifying the boundary of the tree $\T_p$ with $\P^1(\Q_p)$, the splitness condition on $p$ allows us to see $\H'$ as boundary elements of $\T_p$. For $\tau$ in $\H'$, we define in Section \ref{IScycles} a cycle $\Delta_\tau$ inside the space $T^{(p)}(Y_0(\C))$ than depends only on the $\Gamma^+$-orbit of $\tau$. The projection of $\Delta_\tau$ on $\Gamma^+\backslash\H$ is just $\Gamma^+\tau$. The fiber over $\Gamma^+\tau$ under the natural projection $\lquot{\Gamma^+}{(\H'\times\T_p)}\to \lquot{\Gamma^+}{\H'}$ is $\lquot{\Gamma^+}{(\Gamma^+\tau\times \T_p)}$. This is a graph having the structure of a $(p+1)$-volcano, that can be identified with the $p$-isogeny graph attached to any elliptic curve coming from a $\PSL_2(\Z)$-orbit in $\Gamma^+\tau$. The projection of $\Delta_\tau$ to this graph is its unique closed loop (known as its rim). Observe that there's a bijection between $\lquot{\Gamma^+}{(\Gamma^+\tau\times\T_p)}$ and $\{\tau\}\times\lquot{\Gamma^+_\tau}{\T_p}$, where $\Gamma^+_\tau$ denotes the stabilizer of $\tau$ in $\Gamma^+$. The aforementioned loop is formed by taking the unique geodesic in $\T_p$ joining $\tau$ and its Galois conjugate and projecting it to $\lquot{\Gamma^+_\tau}{\T_p}$. This resembles the complex case where the closed geodesics on $Y_0(\C)$ are projections of geodesics in $\H$ joining two real quadratic numbers that are conjugated. In \cite{BDIS}, such loops in $\lquot{\Gamma^+_\tau}{\T_p}$ are called Shintani-cycles, for that reason we refer to $\Delta_\tau$ as the Ihara-Shintani cycle attached to $\Gamma^+\tau$. In this way, we think of the cycles $\Delta_\tau$ as the $p$-adic analogues to the compact geodesic orbits on the modular surface and the space $T^{(p)}(Y_0(\C))$ as a $p$-adic unit tangent bundle of $Y_0(\C)$. 

For $\tau\in\H'$, we denote by $\O_\tau$ the ring formed by the elements in $M_2(\Z[1/p])$ having the column $(\tau\,\,\, 1)^t$ as an eigenvector. Let $K$ be a quadratic imaginary field at which $p$ splits and let $\O$ be a $\Z[1/p]$-order inside $K$. We denote by $IS(\O)$ the collection of $\Gamma^+$-orbits of those $\tau\in\H'$ for which $\Q(\tau)=K$ and $\O_\tau$ is isomorphic to the order $\O$. The group $Pic(\O)$ act simply transitive on $IS(\O)$ and so this set is finite. Let $A$ be the diagonal group of $\PGL_2$. The cycle $\Delta_\tau$ comes equipped with a unique $A(\Q_p)$-right invariant probability measure $\nu_\tau$. Denote by $\mu_\O$ the unique probability measure proportional to $\sum_{\Gamma^+\tau\in IS(\O)}\nu_\tau$. 

Equip $T^{(p)}(Y_0(\C))$ with the unique probability measure coming from a Haar measure on $\PGL_2(\R\times\Q_p)$ and the counting measure in $\Gamma$. Denote this measure by $\mu$. Our $p$-adic analogue to Duke's theorem on the equidistribution of closed geodesics is the following statement.

\begin{customtheorem}{A}[]\label{thm1}
  As $disc(\O)\to-\infty$, the sequence $\mu_\O$ converges to $\mu$ in the weak-* topology. In other words, for every continous and compactly supported function $f\colon T^{(p)}(Y_0(\C))\to\C$, \[ \lim_{disc(\O)\to-\infty}\int fd\mu_\O\to\int fd\mu.\]
\end{customtheorem}

Now we discuss our result on Heegner points on Shimura curves. Fix $\mathcal{B}$ a non-split indefinite quaternion algebra defined over $\Q$ whose discriminant $N^-$ is divisible by $p$. Also fix $\mathcal{R}$ an Eichler-order in $\mathcal{B}$ of level $N^+$. Let $X$ be the Shimura curve attached to $(\mathcal{B},\mathcal{R})$ as in section \ref{Heeg}. Let $\Q_{p^2}$ be the unique quadratic unramified extension of $\Q_p$. The compact space $X(\Q_{p^2})$ admits a uniformization by $\PGL_2(\Q_p)$ due to Cerednik and Drinfeld. Pushing a Haar measure of $\PGL_2(\Q_p)$ we can endow $X(\Q_{p^2})$ with a natural probability measure $\mu_X$. Let $\O$ be an order in a quadratic imaginary field in which $p$ is inert and whose discriminant is coprime to $pN^-N^+$. Let $Heeg(\O)$ denote the collection of Heegner points in $X(\Q_{p^2})$ with endomorphism ring isomorphic to $\O$ defined in Section \ref{Heeg}. We obtain the following equidistribution theorem.

\begin{customtheorem}{B}[]\label{thm2}
  As $disc(\O)\to-\infty$, the collection $Heeg(\O)$ becomes equidistributed on $X(\Q_{p^2})$ with respect to the measure $\mu_X$. In other words, for every continuous function $f\colon X(\Q_{p^2})\to\C$, \[\lim_{disc(\O)\to-\infty}\frac{1}{\#Heeg(\O)}\sum_{P\in Heeg(\O)}f(P)\to\int_{X(\Q_{p^2})}fd\mu_X.\]
\end{customtheorem}

A related result by Disegni \cite{Disegni} concerns the case where the power of $p$ in the conductor of the Heegner points tends to infinity. In such a situation there is no accumulation measure on $X(\C_p)$ describing the asymptotic distribution of their Galois orbits

The main tool for our purposes is Theorem 4.6 in \cite{ELMV3}. In consequence, our strategy is to describe the previous objects and ambient spaces in an $S$-arithmetic and adelic context. In doing so we will appeal to the language of (oriented) optimal embeddings to parametrize our packets as orbits under certain Pic groups which will lead to a description of these as toric orbits inside a homogeneous space. In section \ref{thetree} we cover the basics about the Bruhat-Tits tree of $\PGL_2$ making special emphasis on how quadratic algebras act on it. In section \ref{adelic} we recall the language of adelic and $S$-arithmetic homogeneous spaces as presented for instance in \cite{ELMV3}. In section \ref{cycles} we elaborate on the definition of cycles $\Delta_\psi$ in a general setting, where $\psi$ runs over embeddings of quadratic extension to a quaternion algebra. More specifically, in section \ref{Eorders} and \ref{Oembeddings} we discuss the structure of Eichler orders and oriented optimal embeddings respectively. In section \ref{Picaction} we show how the group $Pic(\O)$ acts simply and transitively on the set of oriented optimal embeddings (with respect to $\O$), showing that the cycles $\Delta_\psi$ defined in section \ref{cyclesdef} are finite and they can be recovered from a single oriented optimal embedding. Section \ref{IScycles} is a specialization of the previous sections and it is where we properly define the cycles $\Delta_\tau$ for $\tau$ in $\Gamma^+\backslash \H'$ and discuss the analogies with the geodesics in Duke's theorem. Section \ref{Heeg} deals with the different specialization related to Heegner points on the $p$-adic points of indefinite Shimura curves. Finally, section \ref{adelicmethod} presents Theorem 4.6 in \cite{ELMV3}. We explain how to apply it to show the equidistribution of the objects defined in sections \ref{IScycles} and \ref{Heeg}, proving in particular Theorem \ref{thm1} and Theorem \ref{thm2}.

\section*{Acknowledgments}

I would like to thank my advisor Ricardo Menares for trusting me with this project and for his support throughout this work. I also would like to thank Philippe Michel and the Analytic Number Theory group at EPFL. I was very lucky to enjoy their hospitality during the academic period 2022-2023. This work was supported by ANID Doctorado Nacional No 21200911 and by the Bourse d’excellence de la Confédération suisse No. 2022.0414.

\section{Dynamics of tori acting on the Bruhat-Tits tree}\label{thetree}

Our main references for this section are \cite{Serre}, \cite{Voight} and \cite{Cass}. Let $\ell$ be a rational prime. Let $\T_\ell$ be the Bruhat-Tits tree of $\PGL_2(\Q_\ell)$. Its set of vertices $\mathcal{V}(\T_\ell)$ corresponds to homothety classes of $\Z_\ell$-lattices in $\Q_\ell^2$. The class of a lattice $\Lambda$ is denoted by $[\Lambda]$ and for $n\in\Z$ we set $v_n\colonequals[\ell^n\Z_\ell\times\Z_\ell]$. Two vertices are connected by an edge if they admit representatives $\Lambda$ and $\Lambda'$ such that $\ell\Lambda\subseteq \Lambda'\subseteq \Lambda$ and $\Lambda/\Lambda'$ is cyclic of order $\ell$. With these definitions, $\T_\ell$ is a $(\ell+1)$-regular tree. 

If $e$ is an edge connecting two vertices, the choice of one of them as its source (denoted $s(e)$) and the other one as its target (denoted $t(e)$) make $e$ an oriented edge. We denote by $\overrightarrow{\mathcal{E}}(\T_\ell)$  the set of oriented edges. Let $e_\infty$ be the oriented edge such that $s(e_\infty)=v_0$ and $t(e_\infty)=v_{-1}$. The group $\PGL_2(\Q_\ell)$ acts transitively on $\mathcal{V}(\T_\ell)$ and $\overrightarrow{\mathcal{E}}(\T_\ell)$. The subgroups $\PGL_2(\Z_\ell)$ and \[\Gamma_0(\ell\Z_\ell)=\left\{\begin{pmatrix}a&b\\c&d\end{pmatrix}\in \PGL_2(\Z_\ell): \ell\mid c\right\}\] are the stabilizer of $v_0$ and $e_\infty$ respectively. Consequently, we have identifications \[\mathcal{V}(\T_\ell)=\PGL_2(\Q_\ell)/\PGL_2(\Z_\ell) \mbox{ and }\overrightarrow{\mathcal{E}}(\T_\ell)=\PGL_2(\Q_\ell)/\Gamma_0(\ell\Z_\ell).\]

A path of length $n\geq0$ in $\T_\ell$ is sequence of $n$ adjacent vertices without backtracking. Paths of length $n$ starting from $v_0$ are identified with $\P^1(\Z/\ell^n\Z)$: To give such a path is equivalent to give a vertex $v$ at distance $n$ from $v_0$. The group $\PGL_2(\Z_\ell)$ acts transitively on such vertices and one such vertex is $v_{-n}$. Its stabilizer in $\PGL_2(\Z_\ell)$ is \[\Gamma_0(\ell^n\Z_\ell)=\left\{\begin{pmatrix}a&b\\c&d\end{pmatrix}\in \PGL_2(\Z_\ell): \ell^n\mid c\right\}.\] Therefore we have $\PGL_2(\Z_\ell)$-equivariant identifications between \begin{equation}\label{neighbords}\PGL_2(\Z_\ell)/\Gamma_0(\ell^n\Z_\ell)\cong\P^1(\Z/\ell^n\Z),\end{equation} and paths of length $n$ starting from $v_0$ such that the path ending at $v_{-n}$ corresponds to $\infty$. The action of $\PGL_2(\Z_\ell)$ on the right hand side of \eqref{neighbords} is by fractional linear transformations.

A branch in $\T_\ell$ is a sequence of adjacent vertices without backtracking. Two branches are equivalent if they differ by a finite initial sequence. An equivalence class of branches is called an end of $\T_\ell$ and we denote the set of ends by $\partial\T_\ell$. We can see $\partial\T_\ell$ as the boundary of the tree and identify it with $\P^1(\Q_\ell)$: The group $PGL_2(\Q_\ell)$ acts transitively on the set of ends and the stabilizer of the end coming from the branch $\{v_{-n}\}_{n\geq0}$ is the subgroup $B(\Q_\ell)\subseteq \PGL_2(\Q_\ell)$ of upper triangular elements. We have $\PGL_2(\Q_\ell)$-equivariant identifications \begin{equation}\label{boundary}\partial\T_\ell\cong\PGL_2(\Q_\ell)/B(\Q_\ell)\cong\PGL_2(\Z_\ell)/B(\Z_\ell)\cong\P^1(\Z_\ell)=\P^1(\Q_\ell),\end{equation} such that that the branch coming from $\{v_{-n}\}_{n\geq0}$ corresponds to $\infty$. As before, under this identification the action of $\PGL_2(\Q_\ell)$ on $\P^1(\Q_\ell)$ corresponds to the action by fractional linear transformations. 

A geodesic is the path without backtracking in the tree connecting two different points in the boundary. The geodesic joining the points $0$ and $\infty$ is given by the vertices $\{v_n\}_{n\in\Z}$ and we denote it by $\mathcal{G}$. In general $\mathcal{G}_x^y$ denotes the geodesic joining $x$ and $y$.

Let $v$ and $v'$ be two vertices in $\T_\ell$. Being $\T_\ell$ a tree, there exists a unique finite path connecting $v$ and $v'$. The distance $d(v,v')$ between $v$ and $v'$ is the length of this path. For instance, two vertices are at distance $1$ if they are connected by an edge. Let $e$ an edge and $\eta$ a path (no necessarily of finite length) in the tree. The distance $d(v,\eta)$ between $v$ and $\eta$ is $\min_{w\in\eta}d(v,w)$ and the distance $d(e,\eta)$ between $e$ and $\eta$ is $\max_{w\in e}{d(w,\eta)}$.

We denote by $A\subseteq PGL_2(\Q_\ell)$ the subgroup of diagonal elements.

\begin{proposition}\label{action1} The following hold
\begin{enumerate}
\item The group $A$ fixes the points $0$ and $\infty$ in $\partial\T_\ell$. For every $n\geq0$, it acts transitively on the set of vertices $v$ with $d(v,\mathcal{G})=n$.
\item The group $\PGL_2(\Z_\ell)$ has $v_0$ as a unique fixed vertex and for every $n\geq0$ it acts transitively on the sets of vertices $v$ with $d(v,v_0)=n$.
\item The group $\left\langle\Gamma_0(\ell\Z_\ell),\begin{pmatrix}0&1\\\ell&0\end{pmatrix}\right\rangle$ has a unique fixed edge $e_\infty=\{v_0,v_{-1}\}$. For every $n\geq0$, it acts transitively on the set of edges $e$ with $d(e,e_\infty)=n$.
\end{enumerate}
\end{proposition}

\begin{proof} See chapter I, sections 3-5 in \cite{Cass}.
\end{proof}

Let $\mathcal{K}\subseteq M_2(\Q_\ell)$ be a quadratic $\Q_\ell$-algebra. We denote by $\O_{\mathcal{K}}$ the maximal compact and open subring of $\mathcal{K}$. The order of conductor $\ell^n$ is $\Z_\ell+\ell^n\O_{\mathcal{K}}$. We set $\mathfrak{m}_\mathcal{K}\colonequals \ell\O_\mathcal{K}$, $\O_{\mathcal{K}}^{(0)}=\O_{\mathcal{K}}^\times$ and for $n\geq1$ we define $\O_\mathcal{K}^{(n)}=1+\mathfrak{m}_{\mathcal{K}}^n$, the $n$-th principal subgroup of units.

\begin{corollary}\label{dynamics} Let $\mathcal{K}\subseteq M_2(\Q_\ell)$ be a quadratic $\Q_\ell$-algebra. Consider the action of $\overline{\mathcal{K}}^\times\colonequals\mathcal{K}^\times/\Q_\ell^*$ on $\T_\ell$ through the embedding  $\mathcal{K}^\times/\Q_\ell^*\subseteq PGL_2(\Q_\ell)$.
\begin{enumerate}
\item If $\mathcal{K}/\Q_\ell$ splits, then $\overline{\mathcal{K}}^\times$ fixes two points $x,y\in \partial\T_\ell$. For $n\geq0$, $\overline{\mathcal{K}}^\times$ acts transitively on the sets of vertices $w$ with $d(w,\mathcal{G}_x^y)=n$. Moreover, the stabilizer of $w$ is the image in $\overline{\mathcal{K}}^\times$ of the order of conductor $\ell^n$ in $\mathcal{K}$ where $d(v,\mathcal{G}_x^y)=n$.
\item If $\mathcal{K}/\Q_\ell$ is an unramified field extension, then $\overline{\mathcal{K}}^\times$ has a unique fixed vertex $v$ and for $n\geq0$, it acts transitively on the sets of vertices $v'$ with  $d(v',v)=n$. Moreover, if $d(v',v)=n$, the stabilizer of $v'$ is the image in $\overline{\mathcal{K}}^\times$ of $\O_\mathcal{K}^{(n)}$.
\item If $\mathcal{K}/\Q_\ell$ is a ramified field extension, then $\overline{\mathcal{K}}^\times$ has a unique fixed edge $e$ and for $n\geq0$, it acts transitively on the set of edges $e'$ with  $d(e,e')=n$. Moreover, if $d(e',e)=n$, the stabilizer of $e'$ is the image of the subgroup $\O_\mathcal{K}^{(2n)}$.
\end{enumerate}
\end{corollary}

\begin{proof} The first item follows directly from Proposition \ref{action1} since there exists $g\in PGL_2(\Q_\ell)$ such that $\overline{\mathcal{K}}^\times=gAg^{-1}$ and in particular $\mathcal{G}_x^y=g\mathcal{G}$. Also working up to conjugation, for $(2)$ and $(3)$ it is enough to prove it for the algebra $\mathcal{K}$ (isomorphic to $\Q_\ell(\sqrt{d})$) given by elements of the form $\begin{pmatrix}x&y\\dy&x\end{pmatrix}$ where $d\in\Z$ is a non-zero square mod $\ell$ in case (2) and for case (3) $\ell\mid\mid d$. Now we need to observe that in case (2) (resp. case (3)) $\overline{\mathcal{K}}^\times$ is contained in $\PGL_2(\Z_\ell)$ (resp.  $\left\langle\Gamma_0(\ell\Z_\ell),\begin{pmatrix}0&1\\\ell&0\end{pmatrix}\right\rangle$). The statement about the fixed elements follows immediately from Proposition \ref{action1}. 

In case (2), transitivity follows since $\overline{\mathcal{K}}^\times$ acts transitively on $\P^1(\Z/\ell^n\Z)$ for every $n$. For (3) one can assume $\ell\mid\mid d$ and note that $\begin{pmatrix}0&1\\d&0\end{pmatrix}$ preserves $e_\infty$ by interchanching $v_0$ and $v_{-1}$. Therefore, it is enough to show that $\overline{\mathcal{K}}^\times\cap \Gamma_0(\ell\Z_\ell)$ acts transitively on the vertices in the connected component of $\T_\ell-e_{\infty}$ connected to $v_{-1}$ and at a given distance to $v_{-1}$. Now for $\begin{pmatrix}x&y\\dy&x\end{pmatrix}$ in the given intersection, $x\in\Z_\ell^\times$ and its action on $(1\,\,\,0)^t$ is given by $(1\,\,\,\, dyx^{-1})^t$. Hence this orbit in $\P^1(\Z/\ell^{n+1}\Z)$, which corresponds to the orbit of $v_{-(n+1)}$ in $\T_\ell-e_{\infty}$, has $\ell^n$ distinct elements which is exactly the number of vertices at distance $n$ from $v_{-1}$ in this component.

Once transitivity is proven, the fact about stabilizers follows since these algebras are commutative and by a counting argument: $\#\rquot{(\O_\mathcal{K}^\times/\O_\mathcal{K}^{(n)})}{(\Z_\ell^*/\Z_\ell^{(n)})}=\ell^{n-1}(\ell+1)$ in case (2):

\[\O_\mathcal{K}^\times/\O_\mathcal{K}^{(n)}\cong\mathbb{F}_{\ell^2}^*\times\O_\mathcal{K}^{(1)}/\O_\mathcal{K}^{(n)}\] which has cardinality $(\ell^2-1)\ell^{2(n-1)}$. Then the quotient $\#\rquot{(\O_\mathcal{K}^\times/\Z_\ell^\times)}{(\O_\mathcal{K}^{(n)}/\Z_\ell^{(n)})}$ has cardinality $(\ell^2-1)\ell^{2(n-1)}/(\ell-1)\ell^{n-1}=(\ell+1)\ell^{n-1},$ the number of points at distance $n$ from $v_0$.

In case (3), \[(\O_\mathcal{K}^\times/\O_\mathcal{K}^{(2n)})\cong\mathbb{F}_\ell^*\times\O_\mathcal{K}^{(1)}/\O_\mathcal{K}^{(2n)},\] which has cardinality $(\ell-1)\ell^{2n-1}$. Then \[(\O_\mathcal{K}^{(1)}/\Z_\ell^{(1)})/(\O_\mathcal{K}^{(2n)}/\Z_\ell^{(n)})=(\O_\mathcal{K}^{(1)}/\O_{\mathcal{K}}^{(2n)})/(\Z_\ell^{(1)}/\Z_\ell^{(n)}),\] has cardinality $\ell^{2n-1}/\ell^{n-1}=\ell^{n}$, the number of vertices in the connected component of $v_{-1}$ and at distance $n$ to it.
\end{proof}

\section{Adelic and $S$-arithmetic homogeneous spaces}\label{adelic}

Fix $p$ be a rational prime. Let $\A$ be the ring of adeles over $\Q$. Let $B$ be a quaternion algebra over $\Q$ that splits over $p$. Denote by $\GG$ the algebraic group $PB^\times\colonequals B^\times/\mathbb{G}_m$ of projectivized units of $B$. We use $\widetilde{\GG}$ for the algebraic group $B^{1}$ of units of reduced norm $1$ in $B$.

Let $A$ be a ring. If $v$ a place of $\Q$, we denote by $A_v$ the algebra $A\otimes\Z_v$ if $v$ is finite or $A\otimes\R$ if $v=\infty$. Let $S$ be a finite set of places of $\Q$ containing the place $p$. We use the notations $\A^{(S)}=\prod'_{w\not\in S}\Q_w$, $\A_f=\A^{(\infty)}$, $\A_\#=\A^{(p,\infty)}$ and $\Q_S=\prod_{v\in S}\Q_v$. Anagolously , we use $\widehat{\Z}$, $\widehat{\Z}_\#$ to denote the closure of $\Z$ in $\A_f$ and in $\A^{(p)}\subset\A_\#$ respectively. In the same direction, $\widehat{A}$ and $\widehat{A}_\#$ will denote $A\otimes\widehat{\Z}$ and $A\otimes\widehat{\Z}_\#$ respectively. Finally, if $t$ is an adelic element in $\A$, $t_f$, $t_\#$, $t_S$ and $t_v$ will denote the respective projection of $t$ in $\A_f$, $\A_\#$, $\Q_S$ and $\Q_v$. Similarly for $\A$-points over $\GG$ and $\widetilde{\GG}$.

The group $\widetilde{\GG}$ satisfies the following strong approximation theorem.

\begin{theorem}\label{strong} The subgroup $\sG(\Q)\sG(\Q_p)$ is dense in $\sG(\A)$.
\end{theorem}

\begin{proof} See \cite{Vig}, Théoreme 4.3, p.81.
\end{proof}

In particular, if $W\subseteq \sG(\A^{(S)})$ is an open subgroup, then $\sG(\A)=\sG(\Q)\sG(\Q_S)W$.

\begin{corollary}\label{strongdec} Let $H\subseteq B^\times(\A^{(S)})$ be an open subgroup such that $\mathrm{nr}(B^\times(\A))=\mathrm{nr}(B^\times(\Q)B^\times(\Q_S)H)$. Then we also have the decomposition $B^\times(\A)=B^\times(\Q)B^\times(\Q_S)H$. In particular $\GG(\A)=\GG(\Q)\GG(\Q_S)\overline{H}$ with $\overline{H}$ the image of $H$ in $\GG(\A^{(S)})$.
\end{corollary}
\begin{proof} By the decomposition $\mathrm{nr}(B^\times(\A))=\mathrm{nr}(B^\times(\Q)B^\times(\Q_S)H)$, given $t\in\GG(\A)$, there exist $x_\Q\in B^\times(\Q)$, $x_S\in B^\times(\Q_S)$ and $x_H\in H$ such that $x_\Q^{-1}tx_H^{-1}x_{S}^{-1}\in\sG(\A)$. We obtain \[B^\times(\A)=B^\times(\Q)\sG(\A)B^\times(\Q_S)H\] which allows us to conclude since \[\sG(\A)=\sG(\Q)\sG(\Q_S)(H\cap \sG(\A^{(S)}))\subseteq B^\times(\Q)B^\times(\Q_S)H.\]
\end{proof}

Since $\GG$ has no non-trivial $\Q$-characters, the subgroup $\GG(\Q)$ is a lattice in $\GG(\A)$ by Theorem 5.5 in \cite{AGNT94}. We will use the notation $[\GG]$ for the locally compact space of finite measure $\lquot{\GG(\Q)}{\GG(\A)}$. If $\KK$ is a compact subgroup of $\GG(\A)$, we use $[\GG]_{\KK}$ to denote the locally compact space $\rquot{[\GG]}{\KK}$. The symbol  $[\cdot]_{\KK}$ denotes the natural projection $\GG(\A)\to [\GG]_{\KK}$ so that the class of $g\in\GG(\A)$ in $[\GG]_{\KK}$ is $[g]_{\KK}$.

We denote by G the group $\GG(\Q_S)$. Assume that $\KK$ is the projection in $\GG(\A)$ from an open subgroup $H\subseteq B^\times(\A^{(S)})$ satisfying the hypothesis of Corollary \ref{strongdec} and set $\Gamma=\GG(\Q)\cap \KK\subseteq G$. By the conclusion of the cited corollary, given $g\in \GG(\A)$ there exists $x\in\GG(\Q)$ such that $x g_{\#}\in\KK$. Like this we obtain an identification \begin{equation}\label{ident0}[\GG]_{\KK}\cong\lquot{\Gamma}{G},\end{equation} sending $[g]_{\KK}$ to the $\Gamma$-orbit of $xg_S$ with $x$ as before.

\section{The cycles}\label{cycles}
\subsection{Eichler orders}\label{Eorders}

Let $p$ and $B$ be as in the previous section and consider $N^-$ as the product of the primes where $B$ ramifies and $N^+$ as a positive integer coprime to $N^-p$. An Eichler $\Z[1/p]$-order $R$ in $B$ is the intersection of two uniquely determined, not necessarily distinct, maximal $\Z[1/p]$-orders in $B$. The level of $R$ is the index of $R$ inside any of the two maximal orders that define it. By a local global-principle (Theorem 9.1.1 in \cite{Voight}) $R$ is determined by its local components $R_q$ for $q\neq p$ so we focus on these local orders. Being also the intersection of two uniquely determined, not necessarily distinct, maximal $\Z_q$-orders in $B_q$ they are Eichler orders in $B_q$ and its level is defined analogously. 

If $q\nmid N^-$, $B$ splits at $q$ and we can use an identification $B_q\cong M_2(\Q_q)$ and the tree $\T_q$ to understand these type of orders. Indeed, maximal orders in $M_2(\Q_q)$ are in bijection with $\mathcal{V}(\T_q)$ by sending the class of a lattice $\Lambda$ to the order $\mathrm{End}_{\Z_q}(\Lambda)\colonequals\{x\in M_2(\Q_q)\mid x\Lambda\subseteq \Lambda\}$. Consequently, the Eichler order $R_q$ can be identified with a pair of vertices in $\T_q$.  If $R_q$ has level $q^n$, then the two corresponding vertices are at distance $n$. In this fashion, we can identify Eichler orders of level $q^n$ in  $B_q$ with paths in $\T_q$ of length $n$. We denote by $\eta_{R_q}$ the path that corresponds to $R_q$. Once given an orientation to this path (select a source $s(\eta_{R_q})$ and target $t(\eta_{R_q})$), the stabilizer in $\PGL_2(\Q_q)\cong PB_q^\times$ of this oriented path is $\overline{R}^\times_q$ (see \cite{Voight} chapter 23). Since an Eichler order of level $1$ (resp. $q$) correspond a vertex (resp. an edge), we denote in those special cases the associated path as $v_{R_q}$ and $e_{R_q}$, respectively.

\begin{remark} All of this shows that the set of maximal $\Z[1/p]$-orders in $B_q$ ($q\nmid N^-p$) forms the set of vertices of a tree which we also denote by $\T_q$ by an abuse of notation. The action of the group $PB_q^\times=B_q^\times/\Q_q^\times$ on $\T_q$ corresponds to conjugation. 
\end{remark}

\subsection{Oriented optimal embeddings}\label{Oembeddings} 

Fix an Eichler $\Z[1/p]$-order $R$ in $B$. Let $K$ be a quadratic field and $\O$ a $\Z[1/p]$-order in $K$. There exists a unique integer $d$ with $p^2\nmid d$ such that $\O=\O_d[1/p]$, where $\O_d\colonequals\Z\left[\frac{d+\sqrt{d}}{2}\right]$, the $\Z$-order of discriminant $d$. We assume that $(d,N^-)=1$ and if $\ell$ is a prime with $\ell^2\mid N^+$, then $(d,\ell)=1$. An algebra embedding $\psi\colon K\to B$ is said to be optimal (with respect to $\O$ and $R$) if $\psi(K)\cap R=\psi(\O)$. Necessarily $K$ is inert at primes diving $N^-$, split at primes dividing $N^+$ and imaginary in the case that $B$ is definite. 

For $\ell$ a prime, note that $\psi$ induces a morphism (which we denote by the same letter) $\psi\colon K_\ell\to B_\ell$ by tensoring with $\Q_\ell$. For $\ell\nmid N^-$ the map $\psi$ induces an action of $K_\ell^\times$ on $\T_\ell$ through composition with the natural map $B_\ell^\times\to PB_\ell^\times$. If $\overline{K}_\ell^\times$ denotes $K^\times_\ell/\Q_\ell^\times$, there is also an induced action of $\overline{K}_\ell^\times$. If $\ell\mid N^+$, since $\psi(K_\ell)$ is a split algebra over $\Q_\ell$, by Corollary \ref{dynamics} (1), the group $K_\ell^\times$ preserves a geodesic in $\T_\ell$ that will be denoted by $\mathcal{G}_\psi$. In addition, the optimality condition ensures that the stabilizer of $\eta_{R_\ell}$ in $K_\ell^\times$ is $\O_\ell^\times\Q_\ell^\times$.
 
An orientation of $R$ at $\ell\mid N^-$ is a ring morphism $\nu_\ell\colon R_\ell\to\mathbb{F}_{\ell^2}$. There exist two of them and we fix one. An orientation of $R$ at $\ell\mid N^+$ is the choice of a maximal order $S_\ell$ in $B_\ell$ containing $R_\ell$. By the first paragraph, this is equivalent to give an orientation to $\eta_{R_\ell}$. There are also two possibilities and we fix one. We orient $\eta_{R_\ell}$ by declaring that $s(\eta_{R_\ell})=v_{S_\ell}$.

For $\ell\mid N^-$, an orientation of $\O$ at $\ell$ is a ring morphism $\mu_\ell\colon\O\to\mathbb{F}_{\ell^2}$. There are two of them an we fix one. Let $\psi\colon K\to B$ an optimal embedding as before. We say that $\psi$ is oriented at $\ell\mid N^-$ if it respects the orientations of $\O$ and $R$ at $\ell$, i.e. if the diagram 
\begin{center}
\begin{tikzcd}
\O_\ell \arrow[r, "\psi"] \arrow[dr, "\mu_\ell"'] & R_\ell \arrow[d, "\nu_\ell"]\\ 
& \mathbb{F}_{\ell^2}
\end{tikzcd}
\end{center}
is commutative.

Take $\ell\mid N^+$ and assume that $\ell\nmid d$. An orientation of $\O$ at $\ell$ is a group isomorphism $\mu_\ell\colon K_\ell^\times/\O_\ell^\times\Q_\ell^\times\to\Z$. There are two of them and we fix one. Since $\ell\nmid d$, $\O_\ell$ is maximal and by Corollary \ref{dynamics} (1), the path $\eta_{R_\ell}$ is contained in $\mathcal{G}_\psi$. The transformation $T_\psi\colonequals \psi(\mu_\ell^{-1}(1))$ acts transitively on $\mathcal{G}_\psi$ as a translation. We say that $\psi$ is oriented if $T_\psi^n(v_{S_\ell})$ and $v_{S_\ell}$ are the two vertices connected by $\eta_{R_\ell}$. 
 
Keep the assumption that $\ell\mid N^+$ but assume now that $\ell\mid d$. Since $K$ splits at $\ell$, the prime $\ell$ does not divides the square-free part of $d$ so $\ell^{2n}\mid\mid d$ for some $n\geq1$. Remember that by our initial asumptions this implies that $\ell\mid\mid N^+$ so $\eta_{R_\ell}=e_{R_\ell}$ is an edge with stabilizer in $K_\ell^\times$ given by $\O_\ell^\times\Q_\ell^\times=\mathcal{O}_{K_\ell}^{(n)}\Q_\ell^\times$. Since $n\geq1$, the edge $e_{R_\ell}$ is not contained in $\mathcal{G}_\psi$. Its furthest vertex from $\mathcal{G}_\psi$ is located at distance $n$ and by consequence the other extreme is located at distance $n-1$. We will say that $\psi$ is oriented if the furthest vertex is $v_{S_\ell}$.

\begin{remark} In simple words, for $\ell\mid N^+$ bud $\ell\nmid d$, $\psi$ is oriented if $\eta_{R_\ell}$ is pointing in the same direction as the flow defined by the transformation $T_\psi$ in $\mathcal{G}_\psi$. On the other hand, if $\ell\mid d$, being oriented is equivalent to say that $e_{R_\ell}$ is pointing towards $\mathcal{G}_\psi$. 
\end{remark}

\begin{lemma} The group $\overline{R}^\times$ respects orientations when acting by conjugation on optimal embeddings.
\end{lemma}

\begin{proof} Let $\psi$ be an oriented optimal embedding. For $x\in\overline{R}^\times$, clearly $x\psi x^{-1}$ still satisfies the optimality condition so we focus on the orientations. For $\ell\mid N^-$, we show that $x\in PB^\times_\ell$ respects the orientations if and only if $x\in \overline{R}_\ell^\times$. There exits some $j\in B_\ell$ such that $B_\ell=\psi(K_\ell)+j\psi(K_\ell)$ with $j^2=p$ and $\alpha j=j\alpha'$ for every $\alpha\in \psi(K_\ell)$, where we denote by $\alpha'$ the conjugate of $\alpha$ in $\psi_\ell(K_\ell)$. Since $(d,N^-)=1$, $\O_\ell$ is maximal in $K_\ell$ and therefore $R_\ell=\psi(\O_{\ell})+j\psi(\O_{\ell})$ with maximal ideal $\m=\psi(p\O_{\ell})+j\psi(\O_\ell)$ (see \cite{Voight} Theorem 13.3.11). Write $x$ in the form $x=\alpha+j\beta$ with $\alpha$ and $\beta$ in $\psi(K_\ell)$. Then $\overline{x}=\alpha'+j\beta$ and $\mathrm{nr}(x)=N(\alpha)-pN(\beta)$ with $N(\cdot)=N_{\psi(K_\ell)/\Q_\ell}(\cdot)$. This allows us to compute for $a\in\psi(\O_\ell)$ the equality \[xax^{-1}=\frac{1}{N(\alpha)-pN(\beta)}(N(\alpha)a-pN(\beta)a'+j\alpha'\beta(a-a')).\] From this we see that $xax^{-1}$ is congruent to $a$ mod $\m$, if $\alpha\in\psi(\O_\ell)^\times$ and to $a'$ if not. The fact that $\alpha\in\psi(\O_\ell)^\times$ occurs if and only if $x$ belongs to $\overline{R}_\ell^\times$ allows us to conclude in this case.

For $\ell\mid N^+$, we show that $x\in \overline{R}_\ell^\times$ respects orientation. Note that $x{S_\ell}x^{-1}={S_\ell}$ since $\overline{R}_\ell^\times\subseteq\overline{S}_\ell^\times$ so that $xv_{S_\ell}=v_{S_\ell}$ and $x\eta_{R_\ell}=\eta_{R_\ell}$. Also note that $\mathcal{G}_{x\psi x^{-1}}=x\mathcal{G}_\psi$. Since $PB_\ell^\times$ preserves distances, these equalities are enough to show that $x\psi x^{-1}$ is oriented noting that when $\ell\nmid d$, $T_{x\psi x^{-1}}=xT_{\psi} x^{-1}$.
\end{proof}

\subsection{The action of $Pic(\O)$.}\label{Picaction}

Keep the notations from the previous sections. If $q\mid N^-$, there exists a unique maximal order in $B_q$ and therefore $R_q$ is maximal. Now theorem 13.3.11 in \cite{Voight} implies that $\mathrm{nr}(R_q^\times)=\Z_q^\times$ since it contains a copy of the ring of integers of the quadratic unramified extension of $\Q_q$. Let $q\nmid N^-p$, then $B$ splits at $q$ and $q^n$ is the level of $R_q$ if and only if $q^n\mid\mid N^+$. From Proposition 23.4.3 in \cite{Voight}, after choosing an isomorphism of $B_q\cong M_2(\Q_q)$, $R_q$ is conjugate to the order $\begin{pmatrix}\Z_q&\Z_q\\q^n\Z_q&\Z_q\end{pmatrix}$ so we also have $\mathrm{nr}(R_q^\times)=\Z_q^\times$ in this case.

From now on, $S$ will denote the finite set of places $\{\infty,p\}$ if $B$ is indefinite and the singleton $\{p\}$ otherwise. Let $\KK$ to be the image in $\GG(\A)$ of $\widehat{R}_\#^\times$ if $B$ is indefinite or the image of $\widehat{R}_\#^\times B_\infty^\times$ otherwise. Since $\mathrm{nr}(R_q^\times)={\Z}_q^\times$ for $q\neq p$, the open subgroup $\KK$ satisfies the hypothesis in Corollary \ref{strongdec}. Indeed, we have the equalities $\mathrm{nr}(B^\times(\A))=\A^\times$ or $\mathrm{nr}(B^\times(\A))=\A_f^\times\times\R_{>0}$ in the indefinite or definite case respectively (Lemma 13.4.9 in \cite{Voight}). Finally, the Hasse-Schilling theorem (Theorem 14.7.4 in \cite{Voight}) and the decomposition $\A^\times=\Q^\times\widehat{\Z}^\times\R^\times$ implies the claim.

Therefore, the identification \eqref{ident0} holds with $\Gamma=\GG(\Q)\cap\KK=\overline{R}^\times$ and we have \begin{equation}\label{ident}[\GG]_{\KK}\cong\lquot{\Gamma}{\GG(\Q_S)}.\end{equation} We recall that the identification is given by sending $[g]_{\KK}$ to the $\Gamma$-class of $xg_S$, where $x\in \GG(\Q)$ is such that $x g_{\#}\in\KK$.

Let $\TT_K$ be the algebraic torus $\mathrm{res}_{K/\Q}\mathbb{G}_m/\mathbb{G}_m$. By the local-global principle for lattices (\cite{Voight} Theorem 9.1.1) and using that $p$ is invertible in $\O$, we can identify $Pic(\O)$ as \[\lrquot{\TT_K(\Q)}{\TT_K(\A_f)}{\widehat{\O}^\times}=\lrquot{\TT_K(\Q)}{\TT_K(\A_{\#})}{\widehat{\O}^\times_{\#}}=\lrquot{\TT_K(\Q)}{\TT_K(\A)}{\widehat{\O}^\times_{\#}\TT_K(\Q_S)}.\]

We denote by $\mathrm{opt}(\O,R)$ the set of all oriented optimal embeddings $K\to B$ and by $[\mathrm{opt}(\O,R)]$ we mean the set of $\overline{R}^\times$-conjugacy classes in $\mathrm{opt}(\O,R)$. The class of an embedding $\psi$ is denoted by $[\psi]$. Now we define an action of $Pic(\O)$ on $[\mathrm{opt}(\O,R)]$. Let $\a\in Pic(\O)$ and assume that it is related to the class of $t\in \TT_K(\A)$. The embedding $\psi$ induces an adelic map $K\otimes\A\to B\otimes\A$ which in turns induces an injection \[\lrquot{\TT_K(\Q)}{\TT_K(\A)}{\widehat{\O}^\times_{\#}}\to[\GG]_{\KK}=[G].\] Remembering the identification \eqref{ident}, the class of $t\in \TT_K(\A)$ is sent to the $\Gamma$-class of $x\psi(t_S)$, where $x\in \GG(\Q)$ is such that $x\psi(t_{\#})\in\KK$. We define the action of $\a$ on $[\psi]$ as $\a\star[\psi]\colonequals [x\psi x^{-1}]$.

\begin{lemma} The action of $\a$ on the class of $\psi$ is well defined and gives rise to an action of $Pic(\O)$ on $[\mathrm{opt}(\O,R)]$
\end{lemma}

\begin{proof}
It is well defined because the class of $\psi(t_{\#})$ depends only on the class of $\a$ and on the other hand conjugating $\psi$ by an element of $\gamma\in \Gamma$ changes $x$ by $x\gamma^{-1}$ since $\gamma\in \KK$. After conjugating $\gamma\psi\gamma^{-1}$ by $x\gamma^{-1}\gamma$ we obtain again $x\psi x^{-1}$. 

Since $x\in \GG(\Q)$, $x\psi x^{-1}$ is indeed a new embedding that send $K\to B(\Q)$. For the optimality it is enough to show that for $q\neq p$ one has $x\psi(K_q)x^{-1}\cap R_q=x\psi(\O_q)x^{-1}$. We know $\psi(K_q)\cap R_q=\psi(\O_q)$ but since $\mathrm{im}(\psi)$ is commutative, after conjugating by $x\psi(t_q)$ we have $x\psi(K_q)x^{-1}\cap (x\psi(t_q))R_q(x\psi(t_q))^{-1}=x\psi(\O_q)x^{-1}$ and now we use $(x\psi(t_q))R_q(x\psi(t_q))^{-1}=R_q$ since $x\psi(t_q)\in \overline{R}_q^\times$.

To see that $x\psi x^{-1}$ is oriented we have to check locally at $\ell\mid N^-N^+$. If $\ell\mid N^-$, $\psi(\overline{K}_\ell)^\times\subseteq \overline{R}_\ell^\times$ so $x\in\overline{R}_\ell^\times$ and therefore $x\psi x^{-1}$ is oriented. If $\ell\mid N^+$, observe that $\psi(t_q)\mathcal{G}_\psi=\mathcal{G}_\psi$ by defintion of $\mathcal{G}_\psi$. Also, we have $x\psi(t_\ell)\eta_{R_\ell}=\eta_{R_\ell}$ since $x\psi(t_\ell)\in\overline{R}_\ell^\times$. Then, since $x\psi(t_\ell)$ preserves distances, $\eta_{R_{\ell}}$ is at the same distance to $x\mathcal{G}_\psi=\mathcal{G}_{x\psi x^{-1}}$. Since $x\psi(t_\ell)$ preserves the orientation of $\eta_{R_\ell}$, it follows that $x\psi x^{-1}$ is oriented noting in the case $\ell\nmid d$ that $\psi(t_\ell)T_\psi\psi(t_\ell)^{-1}=T_\psi$.

Suppose that $\a$ and $\b$ are related to the adelic elements $t$ and $s$. Consider $x,y\in \GG(\Q)$ such that $x\psi(s_\#), yx\psi(t_\#)x^{-1}\in\KK$. The equality $yx\psi(t_\#s_\#)=yx\psi(t_\#)x^{-1}x\psi(s_\#)\in \KK$ shows that $\a\star(\b\star[\psi])=\a\b\star[\psi]$.

\end{proof}

\begin{proposition}\label{trans} The group $Pic(\O)$ acts simply transitively on $[\mathrm{opt}(\O,R)]$.
\end{proposition}

\begin{proof}
To show that the action is simple take $\gamma\in\Gamma$ such that $\gamma^{-1}x\psi (\gamma^{-1}x)^{-1}=\psi$. In particular $\gamma^{-1}x$ belongs to the normalizer of $\psi(K)^\times$ which is $\psi(K)^\times\cup s\psi(K)^\times$ with any $s\in B^\times(\Q)$ such that $s\psi s^{-1}=\psi\circ\sigma$, where $\sigma$ is the non-trivial $\Q$-automorphism of $K$. Since $\gamma^{-1}x$ normalized each value of $\psi$ we must have $\gamma^{-1}x\in\psi(K)^\times$. See Corollaire 2.3, p.6 in \cite{Vig}.This implies $x=\gamma\psi(t)$ with $t\in K^\times$ and therefore $\psi(tt_{\#})\in \gamma\KK=\KK$. This shows that $\a$ is trivial in $Pic(\O)$ and therefore the action is simple.

Now we show the transitivity. Let $\psi$ and $\psi'$ be two optimal embeddings for $\O$. As a consequence of Skolem-Noether there exists $x\in \GG(\Q)$ such that $\psi'=x\psi x^{-1}$. To prove transitivity we need to show that there exists an adelic element $t$ such that $x\psi(t_{\#})\in \KK$. Since both $\psi$ and $\psi'$ are optimal for $\O$, we have for $q\neq p$ \begin{equation}\label{stab}\psi(\O_q)=\psi(K_q)\cap R_q=\psi(K_q)\cap x^{-1}R_qx.\end{equation}

If $q\mid N^-$, we have $\psi(K_q)^\times\subseteq\overline{R}_q^\times$. Then the condition $x\psi(t_q)\in\overline{R}_q^\times$ is equivalent to $x\in \overline{R}_q^\times$. Since $\psi'=x\psi x^{-1}$ and $\psi$ and $\psi'$ are oriented, we have indeed that $x\in \overline{R}_q^\times$. 

If $q\nmid N^-$ remember that $\psi$ induces an action of $K_q^\times$ on $\T_q$ and in this case $\psi(K_q)$ is a split algebra over $\Q_q$ so that there exists a unique geodesic $\mathcal{G}_\psi$ preserved by this action. The relation \eqref{stab} shows that $\eta_{R_q}$ and $x^{-1}\eta_{R_q}$ have the same stabilizer in $K_q^\times$ under this action. If in addition $q\nmid N^+$, $\eta_{R_q}=v_{R_q}$ and $x^{-1}\eta_{R_q}=x^{-1}v_{R_q}$ are actually vertices and Corollary \ref{dynamics} shows they are at the same distance to $\mathcal{G}_\psi$ (since same stabilizer implies same distance) and then they are in the same orbit. Therefore there exists some $t_q\in K_q^\times$ satisfying $x\psi(t_q)v_{R_q}=v_{R_q}$. This implies $x\psi(t_q)\in \overline{R}_q^\times$. 

Now assume that $q^n\mid\mid N^+$ and $q\nmid d$. The relation (\ref{stab}) and the fact that $\psi$ and $\psi'$ are oriented implies that $\eta_{R_q}$ and $x^{-1}\eta{R_q}$ are contained in $\mathcal{G}_\psi$ and they point in the same direction as the flow defined by $T_\psi$. Therefore there exists some $n\in\Z$ such that $xT_\psi^n\eta_{R_q}=\eta_{R_q}$ as an oriented path. Since $T_\psi^n=\psi(\mu_q^{-1}(n))$ we have that $x\psi(t_q)\in\overline{R}_q^\times$ for any $t_q\in K_\ell^\times$ lying in $\mu_q^{-1}(n)$ when taken mod $\O_q^\times\Q_q^\times$.

Finally if $q\mid\mid N^+$ and $q\mid d$, again the condition of being oriented embeddings and relation (\ref{stab}) implies that the edges $\eta_{R_q}=e_{R_q}$ and $x^{-1}\eta_{R_q}=x^{-1}e_{R_q}$ are at the same distance to $\mathcal{G}_\psi$ and pointing in the same direction relative to $\mathcal{G}_\psi$. Corollary \ref{dynamics} (1) implies that they are in the same orbit under the action of $K_q^\times$ and so there exists $t_q\in K_q^\times$ such that $x\psi(t_q)e_{R_q}=e_{R_q}$ as oriented edges, implying that $x\psi(t_q)\in\overline{R}_q^\times$.
\end{proof}

\subsection{The cycles $\Delta_\psi$}\label{cyclesdef}

Keep the previous notations and fix $H$ a subgroup of $G$ that is the image of the units of an algebra in $B_S$ isomorphic to $K_S$.  Let $\psi\in opt(\O,R)$, we denote by $\Delta_\psi$ the image in $\lquot{\Gamma}{G}$ of a set of the form $\psi(K_S)^\times g$ with $g\in G$ such that $\psi(K_S)^\times=gHg^{-1}$ (such $g$ exists by Skolem-Noether). Since $H$ has index $2^{\#S}$ in its normalizer in $G$, there are only $2^{\#S}$ such sets and they depend only on the class of $\psi$ in $[opt(\O,R)]$. In the following sections we will make use of the action of $G$ on certain spaces to select a single $\Delta_\psi$ among the $2^{\#S}$ possibilities.

Let $h'=h'(\O)$ be the cardinality of $Pic(\O)$ and consider $t_1,...,t_{h'}$ ideles in $\widehat{K}^\times_\#\subseteq\A_K^\times$ such that $\widehat{K}^\times_\#=\bigsqcup_{i=1}^{h'}K^* t_i^{-1}\widehat{\O}^\times_\#$. Note that this is equivalent to ask for a full set of representatives for $Pic(\O)$ under its identification with $\lrquot{K^\times}{\widehat{K}^\times_\#}{\widehat{\O}_\#^\times}$. 

\begin{proposition}\label{projection} Fix $\psi_0\in opt(\O,R)$. Then, the projection of $\TT_{\psi_0} g$ in $[\GG]_{\KK}=\lquot{\Gamma}{G}$ is of the form $\bigsqcup_{\psi\in[opt(\O,R)]}\Delta_\psi$.
\end{proposition}

\begin{proof}

The image of $[\TT_K]$ under $\psi_0$ in $[\GG]_{\KK}$ factors through \[\lrquot{K^\times}{\A_K^\times}{\widehat{\O}^\times_\#\Q_S^*}=\bigsqcup_{i=1}^{h'}K^*(t_i^{-1}\widehat{\O}^\times_\#\times (K_S)^\times).\]

The image of $t_i^{-1}\widehat{\O}^\times_\#\times (K_S)^\times$ under $\psi_0$ in $[\GG]_{\KK}$ is $\psi_0(t_i^{-1})\KK\times (\psi_0(K_S))^\times$. Under the identification in (\ref{ident}) this is \[\Gamma x_i\psi_0(K_S)^\times g=\Gamma x_i\psi_0(K_S)^\times x_i^{-1} x_i g,\] where $x_i\in \GG(\Q)$ is such that $x_i\psi_0(t_i^{-1})\in \KK$. This finishes the proof since Proposition \ref{trans} shows that the collection $x_i\psi_0 x_i^{-1}$ runs over a complete set of representatives for $[opt(\O,R)]$.
\end{proof}

\section{Ihara-Shintani Cycles}\label{IScycles}

In this section we specialize to the case $B=M_2(\Q)$, $S=\{p,\infty\}$, $R=M_2(\Z[1/p])$ and $K\subseteq\C$ is imaginary and splits at $p$. In this case $\KK=\PGL_2(\widehat{\Z}_\#)$ and $\Gamma=\PGL_2(\Z[1/p])$ so we have natural identifications \[[\GG]_{\KK\cdot{\mathrm{PSO}_2(\R)}}\cong\lquot{\Gamma}{(\C-\R)\times \PGL_2(\Q_p)}\cong\lquot{\Gamma^+}{\H\times \PGL_2(\Q_p)},\] where $\Gamma^+=\Gamma\cap \PGL_2^+(\R)$.

Let $\mathfrak{P}$ be a prime above $p$ in $\overline{\Z}$, the integral closure of $\Z$ in $\overline{\Q}$. This defines an embedding $\overline{\Q}\hookrightarrow \C_p$ and in particular $K\hookrightarrow \Q_p$. 

As in section 4.2, let $\O$ denotes a $\Z[1/p]$ order in $K$. For any $\psi\in opt(\O,R)$, the torus $\psi(K_\infty)^\times$ acts on $\C-\R$ by fractional linear transformations, having two fixed points conjugate to each other. We denote by $\tau$ its fixed point in $\H\subseteq\C-\R$.  On the other hand, the group $\psi(K_p)^\times$ has two fixed points in $\P^1(\Q_p)$ (when identified with $\partial\T_p$, these are the two end points of the geodesic preserved by Corollary \ref{dynamics}) which are also $\tau$ and its conjugate, since $\psi(K)^\times$ fixes them.

Let $A$ denote the diagonal group of $\PGL_2$ and consider $H=\mathrm{PSO}_2(\R) \times A(\Q_p)$. Let $g\in G$ be such that $\psi(K_S)^\times=gHg^{-1}$ (it exists by Skolem-Noether). It is always the case that both $g_\infty\cdot i$ and $g_p\cdot0$ assume the value $\tau$ or $\tau'$. If we multiply $g$ by some element in $N_G(H)$, the normalizer of $H$ in $G$, we cover all $4=[N_G(H):H]$ possibilities. Then we normalize the choice of $\Delta_\psi$ in $\Gamma\backslash G$ by requiring that $g_\infty\cdot i=\tau$ and $g_p\cdot0=\tau$. 

For $\tau$ in $\H'$, we take $K_\tau$ to be the $\Q$-algebra of matrices $g\in M_2(\Q)$ such that $g\cdot\tau=\tau$, together with the zero matrix. The map sending $g=\begin{pmatrix}A&B\\C&D\end{pmatrix}$ to its eigenvalue $C\tau+D$ (with eigenvector $(\tau\,\,\, 1)^t$) is an isomorphism onto $\Q(\tau)$. Denote its inverse as $\psi_\tau$ and let $IS(\O)$ be as in the Introduction. If $\Gamma^+\tau$ belongs to $IS(\O)$, the map $\psi_\tau$ is an element in $opt(\O,R)$ whose fixed point in $\H$ is $\tau$. We could have also taken the embedding coming from the eigenvalue $C\tau'+D$ but this one is $\Gamma$-conjugate to $\psi_\tau$ since there exists $g\in \PGL_2(\Z)\subset\Gamma$ taking $\tau$ to $\tau'$. 

For $\tau$ in ${\H'}$, we define the Ihara-Shintani cycle $\Delta_\tau$ attacthed to $\Gamma^+\tau$ to be the projection of $\Delta_{\psi_\tau}$ to $[\GG]_{\KK\cdot PSO_2(\R)}$. It corresponds to the $\Gamma^+$ orbit of $\{\tau\}\times K_{\tau,p}^\times g$ in $\H\times \PGL_2(\Q_p)$. If $\Gamma^+\tau\in IS(\O)$, the map $K_p^\times\to \Delta_\tau$ sending $x$ to the $\Gamma^+$-class of $(\tau,\psi_\tau(x)g)$ induces a uniformization $\lrquot{\O^\times}{K_p^\times}{\Q_p^\times}\cong \Delta_\tau$. The measure $\nu_\tau$ from the introduction is the push-forward  of a finite measure in $\lrquot{\O^\times}{K_p^\times}{\Q_p^\times}$ coming from a Haar measure in $K_p^\times$. 

Let $d$ be as in section 4.2. Suppose that $\Gamma^+\tau\in IS(\O)$. Since $\tau\in\H'$ and $K\cong K_\tau$ we have that $p\nmid d$ and $p$ splits into two primes inside $\O_d$. Both of these primes are proper so that they define elements in $Cl(\O_d)$. Let $k$ be the order of $\p=\mathfrak{P}\cap\O_d$ in this group so we can write $\p^k=u\O_d$ for some $u\in\O_d$. 

\begin{lemma}\label{units} We have $\O^\times=\langle\O_d^\times,u,p\rangle$.
\end{lemma}
\begin{proof}

Note that $x\in\O^\times$ if and only if $p^nx\in\O_d$ and $p^m\in x\O_d$ for some $n,m\geq0$. Therefore, after multiplying by a power of $p$, we can assume that $x\in\O_d$ and $x\O_d$ is an ideal dividing $p^m\O_d$, so that it is of the shape $\p^i\p'^j$ with $i,j\geq0$ by uniqueness of the prime decomposition (principal ideals are proper). Here $\p'$ denotes the conjugate of $\p$ under the non-trivial automorphism of $K$. If $i,j\geq k$ we can divide $x$ by a power of $u$ or $u'$ in order to have $0\leq i,j<k$. If $i=j$ we see that $x$ differs from an element in $\O_d^\times$ by a power of $p$. If this is not the case, after taking classes in $Cl(\O_d)$ we obtain a contradiction with the minimality of $k$. All of this proves that $\O^\times=\langle \O_d^\times, u,p\rangle$ since $uu'=p^k$.
\end{proof}

Denote the stabilizer of $\tau$ in $\Gamma^+$ as $\Gamma_\tau^+$. Then $\Gamma_\tau^+=\O_\tau^\times$. Idenfity $K_\p=\Q_p$ and $K_p\cong\Q_p\times\Q_p$ by sending $x\in K$ to $(x,x')$. When $\O_d^\times=\{\pm1\}$, Lemma \ref{units} allows us to see that $\Gamma_\tau^+$ is isomorphic to $(u/u')^\Z$ under $\psi_\tau$. Observe that $\mathrm{ord}_p(u/u')=\mathrm{ord}_\p(u/u')=\mathrm{ord}_\p(u)=k>0$.

\begin{proposition}\label{volcano} For $\tau\in\H'$, the image of $\{\tau\}\times \T_p$ in $\lquot{\Gamma^+}{(\H\times \T_p)}$ has the structure of a $(p+1)$-volcano. It can be identified with the $p$-isogeny graph of the elliptic curve with ordinary reduction associated with $\PSL_2(\Z)\tau$.
\end{proposition}

\begin{proof} Assume $\Gamma^+\tau\in IS(\O)$. The image is $\lquot{\Gamma^+}{\Gamma^+(\{\tau\}\times \T_p)}\approx\{\tau\}\times\lquot{\O_\tau^\times}{\T_p}$ and point (1) in Corollary \ref{dynamics} together with the previous paragraph implies that $\lquot{\O_\tau^\times}{\T_p}$ is a $(p+1)$-volcano with a cycle of lenght $k$ as a rim. This rim is the projection of the geodesic preserved by $K_p^\times$ and the levels are given by distance to the geodesic.

For the second claim we work first with the set of vertices. We have an identification between $\lquot{\Gamma^+}{(\H\times \PGL_2(\Q_p))}$ and $\lquot{\mathrm{PSL}_2(\Z)}{(\H\times \PGL_2(\Z_p))}$ by sending $\Gamma^+(\tau,g)$ to $\PSL_2(\Z)(\gamma\tau,\gamma g)$, where $\gamma\in\Gamma^+$ satisfies $\gamma g\in \PGL_2(\Z_p)$. Therefore we can also identify \begin{equation}\label{projection}\lquot{\Gamma^+}{(\H\times\mathcal{V}(\T_p))}\cong\lrquot{\Gamma^+}{(\H\times \PGL_2(\Q_p))}{\PGL_2(\Z_p)}\cong\lquot{\PSL_2(\Z)}{\H}.\end{equation}

Consider $\Gamma^+(\tau,v)\in \lquot{\Gamma^+}{(\H'\times\mathcal{V}(\T_p))}$. There exist $g\in \PGL_2(\Q_p)$ such that $v=gv_0$. Under the previous identifications this orbit corresponds to the elliptic curve associated with $\PSL_2(\Z)\gamma\tau$, where $\gamma g\in \PGL_2(\Z_p)$. Certainly this elliptic curve is $p^m$-isogenous to the elliptic curve attached to $\PSL_2(\Z)\tau$, for some $m\geq0$, and it has CM by $\O_{dp^{2n}}$ for some $n\geq0$. This is the case if $\psi_{\gamma\tau}(K)\cap M_2(\Z)=\psi_{\gamma\tau}(\O_{dp^{2n}})$. But $\psi_{\gamma\tau}=\gamma\psi_\tau\gamma^{-1}$ so this condition boils down to $\psi_\tau(K)\cap\gamma^{-1}M_2(\Z)\gamma=\psi_\tau(\O_{dp^{2n}})$. We can check this condition locally. Since $\gamma\in \PGL_2(\Z_q)$ for $q\neq p$, it is only neccesary to check that $\psi_\tau(K_p)\cap\gamma^{-1}M_2(\Z_p)\gamma=\psi_\tau(\O_{dp^{2n}}\otimes\Z_p)$. Since the projectivized units of $\gamma^{-1}M_2(\Z_p)\gamma$ form the stabilizer of $\gamma^{-1}v_0=g v_0$ and $\O_{dp^{2n}}\otimes\Z_p$ is the order of conductor $p^n$ in $\O_d\otimes\Z_p$, by Corollary \ref{dynamics} this is equivalent to $v=gv_0$ being at distance $n$ from the geodesic fixed by $K_p^\times$. 

For the edges, we remark that \[\lquot{\Gamma^+}{(\H\times\overrightarrow{\mathcal{E}}(\T_p))}\cong\lrquot{\Gamma^+}{(\H\times \PGL_2(\Q_p))}{\Gamma_0(p\Z_p)}\cong\lquot{\Gamma_0(p)}{\H}.\] This shows that two points in $\Gamma^+\backslash(\Gamma^+\tau\times\T_p)$ are connected by an edge if and only if the respective elliptic curves are $p$-isogenous. This ends the proof of the second claim.
\end{proof}

By definition of $\Delta_\tau$, the vertex $gv_0$ belongs to the geodesic preserved by $K_p^\times$ so we have that the second coordinate of $\Delta_\tau$ projects to the rim of $\lquot{\O_\tau^\times}{\T_p}$. This closed cycle in the quotient of the tree is what the authors in \cite{BDIS} refer to as a $p$-adic Shintani cycle. This justifies the name of $\Delta_\tau$ and since the rim is a closed cycle we can think of it as a discrete closed geodesic mimicking Duke's geodesics in the archimedean case. Note the following consequence of Proposition \ref{volcano} and its proof.

\begin{corollary}The image of $\Delta_\tau$ in $\lquot{\PSL_2(\Z)}{\H}$ under \eqref{projection} corresponds to the $k$ elliptic curves with CM by $\O_d$ and which are $p^\infty$-isogenous to the elliptic curve associated with $\PSL_2(\Z)\tau$. 
\end{corollary}

The set of elliptic curves described in the above Corollary form a closed orbit under the usual action of $\p^\Z$ on the set of elliptic curves with CM by $\O_d$. In this regard, just as Duke's geodesics form a closed orbit under the geodesic flow, the cycles $\Delta_\tau$ are a closed orbit under a ``flow "defined by $p$.

\section{Heegner points on Shimura curves}\label{Heeg}

In this section we focus on the cases where $B$ is a definite quaternion algebra ramified at $N^-$ and $K$ is imaginary and inert at $p$. Remember that $R$ is a $\Z[1/p]$-Eichler order of level $N^+$ and $\Gamma=\overline{R}^\times$. Denote by $\Gamma^+$ the image in $\Gamma$ of $R_1$, the elements in $R$ with reduced norm equal to $1$.

Let $\mathcal{B}$ be the indefinite quaternion algebra over $\Q$ ramified at the primes dividing $N^-p$. Let $\mathcal{R}$ be an Eichler order of level $N^+$ in $\mathcal{B}$. After fixing an isomorphism $\iota_{\infty}\colon\mathcal{B}_{\infty}\cong M_2(\R)$, the group $\Gamma_\infty=\iota_{\infty}(\mathcal{R}^\times)$ acts on $\mathcal{H}_\infty\colonequals\C\smallsetminus\R$. The space $\lquot{\Gamma_\infty}{\mathcal{H}}_\infty$ is a compact Riemann surface that corresponds to the complex points of an algebraic variety $X$ defined over $\Q$ i.e. $X(\C)=\lquot{\Gamma_\infty}{\mathcal{H}}_\infty$. We refer to $X$ as the Shimura curve attached to the data $(\mathcal{B},\mathcal{R})$.

Let $\Q_{p^2}$ be the unique unramified quadratic extension of $\Q_p$. The compact space $X(\Q_{p^2})$ admits the following $p$-adic uniformization. Let $\H_p\colonequals\Q_{p^2}-\Q_p$, the Drinfeld upper-half plane. After fixing an isomorphism $B_p\cong M_2(\Q_p)$, $\Gamma$ acts on $\H_p$ by fractional linear transformations.

\begin{theorem}[Cerednik-Drinfeld]\label{padicunif} The quotient space $\lquot{\Gamma^+}{\H_p}$ corresponds to the space of $\Q_{p^2}$-points of the Shimura curve $X$.
\end{theorem}
\begin{proof} See Theorem 4.7 in \cite{BD} or also \cite{Cerednik} or \cite{Drinfeld}.
\end{proof}

Fix some $u\in \Z_p^\times\smallsetminus (\Z_p^\times)^2$ and $\tau_0$ a square root of $u$ so that $\Q_{p^2}=\Q_p(\tau_0)$. The stabilizer $\mathbb{O}_p$ of $\tau_0$ in $PB^\times_p$ is compact (isomorphic to $\Q_{p^2}^\times/\Q_p^\times$). Since $\PGL_2(\Q_p)$ acts transitively on $\H_p$, we have the identification of locally compacts spaces \[[\GG]_{\KK\cdot \mathbb{O}_p}\cong\lrquot{\Gamma}{PB_p^\times}{\mathbb{O}_p}\cong\lquot{\Gamma}{\H_p},\] as in (\ref{ident}).

Let $\omega$ be any element in $\Gamma-\Gamma^+$. Since $[\Gamma:\Gamma^+]=2$, it follows that $\omega$ defines an involution on $X(\Q_{p^2})$ and therefore $\lquot{\Gamma}{\H_p}$ corresponds to pairs of points in $X(\Q_{p^2})$ connected by this involution.

According to the moduli interpretation of $X$ (see section 4 in \cite{BD}), the points in $X(\Q_{p^2})$ corresponds to abelian surfaces over $\Q_{p^2}$ with quaternionic multiplication by $\mathcal{R}$ and a $N^+$-level structure. Given $A$, one such abelian surface, we denote by $\mathrm{End}(A)$ the algebra of endormophism of $A$ (over $\overline{\Q}$) which commute with the quaternionic multiplication and respect the $N^+$-level structure. A Heegner point in $X(\Q_{p^2})$ is a point whose associated abelian surface $A$ has $\mathrm{End}(A)$ isomorphic to an order in a quadratic imaginary field.

As in section 4.2, let $\O$ denotes a $\Z[1/p]$ order in $K$. Let $\psi\in opt(\O,R)$. Then $\psi(K_p)^\times$ acts in $\H_p$ with two fixed points $\tau$ and $\tau'$. We assume that $\tau$ is the fixed point satisfying that for every $x\in K_p^\times$, $\psi(x)$ acts on the column vector $(\tau\,\,\,\, 1)^t$ as multiplication by $x$. This time we take $H=\mathbb{O}_p$ and $\Delta_\psi$ equals the $\Gamma$-orbit of $\psi(K_p)^\times g$ with $g\in PB_p^\times$ such that $\psi(K_p)^\times=g \mathbb{O}_pg^{-1}$. We normalize the choice of $g$ by asking that $g\tau_0=\tau$. The projection of $\Delta_\psi$ in $\lquot{\Gamma}{\H_p}$ is $\Gamma\tau$. It corresponds to the pair of points in $X(\Q_{p^2})$ associated with $\Gamma^+\tau$ and $\Gamma^+\omega\tau$. Let $d$ be as in section 4.2.

\begin{theorem} Under the identification given by Theorem \ref{padicunif}, the class of $\Gamma^+\tau$ corresponds to a Heegner point in $X(\Q_{p^2})$ whose associated order is isomorphic to $\O_d$. As $\psi$ varies over $[opt(\O,R)]$, they are all different. \end{theorem}
\begin{proof} Since $p$ is inert in $K$, $Pic(\O)\cong Pic(\O_d)$. Now see Theorem 5.3 in \cite{BD}.
\end{proof}

The previous result justifies that we name by $\mathrm{Heeg}(\O)$, the collection of $\Gamma\tau$ in $\lquot{\Gamma}{\H_p}$ as $\psi$ runs over $[opt(\O,R)]$.\\

\section{Adelic methods}\label{adelicmethod} 

In this last section we prove Theorem \ref{thm1} and Theorem \ref{thm2}. They will be a direct consequence of Proposition \ref{prueba}.

The group $\GG(\A)$ is unimodular (any left Haar measure is also a right Haar measure) so there is a unique $\GG(\A)$-invariant probability measure $dm$ in $[\GG]$ such that for any $f\in C_c(\GG(\A))$, \begin{equation}\label{unfold}\int_{\GG(\A)}fdg=\int_{[\GG]}\int_{\GG(\Q)}f(hg)dhdm(g),\end{equation} where $dh$ is the counting measure in $\GG(\Q)$ and $dg$ a Haar measure on $\GG(\A)$.

\begin{definition} A homogeneous toral subset in $[\GG]$ is a subset of the form $Y=\TT(\Q)\backslash \TT(\A)g$, with $g\in \GG(\A)$ and $\TT\subseteq\GG$ a maximal torus anisotropic over $\Q$.
\end{definition}

The pushforward of the Haar probability measure on $\TT(\Q)\backslash\TT(\A)$ defines a probability measure $\mu_Y$ on the homogeneous toral set $Y=\TT(\Q)\backslash \TT(\A)g$.

Let $K$ be a quadratic field and let $\psi\colon K\to B(\Q)$ be an algebra embedding. Let $\TT_K$ be the algebraic torus $\mathrm{res}_{K/\Q}\mathbb{G}_m/\mathbb{G}_m$. Then $\psi$ induces a morphism (which we denote by the same letter) $\psi\colon\TT_K\to\GG$. Denote by $\TT_\psi$ the image of $\TT_K$ under $\psi$ in $\GG$. Then every maximal anisotropic torus $\TT\subseteq\GG$ defined over $\Q$ is of the form $\TT_\psi$ for some embedding $\psi$. Now, we attach to $\TT_\psi$ an order in $K$. Fix a $\Z$-order $R$ in $B(\Q)$. Then define the local orders $\Lambda_\ell=\psi(K_\ell)\cap g_\ell R_\ell g_\ell^{-1}$. The order $\Lambda_\ell$ is maximal for almost every $\ell$. Indeed, let $T$ be the finite set of places such that $g_\ell\not\in \overline{R}_\ell^\times$, where $\overline{R}_\ell^\times$ denotes the image of $R_\ell^\times$ in $PB_\ell^\times=B_\ell^\times/\Q_\ell^\times$. Then if $\ell\not\in T$, the intersection $\psi(K_\ell)\cap g_\ell R_\ell g_\ell^{-1}$ is just $\psi(K_\ell)\cap R_\ell$. But this is the localization of the global order $\psi(K)\cap R$ and so it is maximal outside a finite set of places. This shows that $\Lambda=K\cap\prod\Lambda_\ell$ is a global order in $K$.

\begin{definition}[\cite{Khay} section 2.4.4 or equivalently \cite{ELMV3}] The discriminant of a homogeneous toral set of the form $Y=\TT_\psi g$, with $\psi\colon K\to B(\Q)$ an embeeding, is the absolute value of the discriminant of the order $\Lambda$ attached to it as in the previous paragraph.
\end{definition}

\begin{theorem}\label{metathm} Let $\{Y_i\}$ be a sequence of homogeneous toral sets whose discriminants approach $\infty$ as $i\to\infty$. Then, any weak* accumulation point of the sequence of measures $\mu_{Y_i}$ is a homogeneous probability measure on $[\GG]$, invariant under $\GG(\A)^+$, the image of $\sG(\A)$ in $\GG(\A)$.
\end{theorem}

\begin{proof} Theorem 4.6 in \cite{ELMV3}
\end{proof}

The defect of the limit measure in Theorem \ref{metathm} being $\GG(\A)$-invariant is solved in the following fashion. Assume that $\{\mu_{Y_i}\}$ weak* converges to $\mu$. We want to prove that for all $f$ in $C_u([\GG])$, the space of bounded and uniformly continuous functions, we have \begin{equation}\label{reduction}\int_{[\GG]}f(g)d\mu(g)=\int_{[\GG]}f(g)dm(g).\end{equation}

By the discussion made \cite{Aka}, sections 9 and 10, we are reduced to prove (\ref{reduction}) in the case of functions invariant under $\GG(\A)^+$. Let $\GG_{char}$ denote the quotient $\rquot{[\GG]}{\GG(\A)^+}$. The reduced norm induces a homeomorphism \[\GG_{char}\cong \lrquot{\Q^\times}{\A^\times}{(\A^\times)^2},\] which gives to $\GG_{char}$ the structure of a compact abelian group. Since the group of characters is dense in the space of continuous functions, by Weyl's criterion we are reduced to the case when $f$ is a character. 

If $\chi$ is a character of  $\lrquot{\Q^\times}{\A^\times}{(\A^\times)^2}$, by continuity there exists some compact open group $M\subseteq \A_f^\times$ such that $\chi\mid_{M}=1$ and therefore we have a character of the group $\lrquot{\Q^\times}{\A^\times}{\R_{>0}M}\cong(\Z/D\Z)^\times$ for some $D\in\Z$ and trivial on squares. Therefore, if we assume $M$ maximal, we are in the presence of a primitive real Dirichlet character that must be attached to some quadratic field $K_\chi$.

We have reduced the test function to be of the shape $\chi\circ\mathrm{nrd}$ where $\chi$ is a quadratic Hecke character.

\begin{proposition}\label{test} Let $\chi$ be a non-trivial character of $\GG_{char}$ and $Y=\TT_\psi g$ an homogeneous toral set with $\psi\colon K\to B$. Then, \[\int_{Y}\chi(\mathrm{nrd}(t))d\mu_{Y}(t)=\chi(g)\int_{\TT_\psi}\chi(N_{K/\Q}(t))d\mu_{\TT_\psi}(t)=\begin{cases}\chi(g)&\mbox{, if }K=K_\chi\\0&\mbox{, if }K\neq K_\chi\end{cases}\] 
\end{proposition}

\begin{proof}
Assume that $\chi$ is related to $K_\chi=\Q(\sqrt{D})$ in the sense of the previous paragraph. If $L$ denotes $\Q(\zeta_{|D|})$, then $K\subseteq L$ is a Galois extension and class field theory provides the diagram
\begin{center}
\begin{tikzcd}
\A_K^\times \arrow[r, twoheadrightarrow] \arrow[d, "N_{K/\Q}"'] & \mathrm{Gal}(L/K) \arrow[d, twoheadrightarrow]\\ 
\A_\Q^\times \arrow[dr, twoheadrightarrow] \arrow[r, twoheadrightarrow, "rec"] & \mathrm{Gal}(L/\Q)\arrow[d, twoheadrightarrow] \arrow[r, "\sim"] & (\Z/D\Z)^\times \arrow[d, "\chi"]\\
& \mathrm{Gal}(K/\Q) \arrow[r, "\sim"]&\{\pm1\}.
\end{tikzcd}
\end{center}

If $K=K_\chi$ and $t\in \TT_{\KK}$ the equality $\chi\circ\mathrm{nrd}(\psi(t))=\chi(N_{K/\Q}(t))=\chi(N_{K_\chi/\Q}(t))=1$ follows from the previous diagram. When $K\neq K_\chi$ there exists a prime $\ell$ split in $K$ and inert in $K_\chi$. Let $s$ be the idele corresponding to the class of a prime $\mathfrak{l}$ above $\ell$ in $K$. Then $\chi(\mathrm{nrd}(\psi(s)))=\chi(N_{K/\Q}(\mathfrak{l}))=\chi(p)=-1$ and so the substitution $t\mapsto ts$ shows that $\int_{\TT_{\psi}}\chi(\mathrm{nrd}(t))d\mu_{\TT_\psi}(t)=0$.  
\end{proof}

Let $M$ be a compact subgroup of $\GG(\A_f)$ with $\mathrm{nrd}(M_p)(\Z_p^\times)^2=\Z_p^\times$ for almost every $p$. Let $\mathcal{K}(M)$ be the finite set of quadratic fields attached to the characters of the finite group $\lrquot{\Q^\times}{\A^\times}{(\A^\times)^2\mathrm{nrd}(M)}$. Equip $[\GG]_M$ with $dm_M$ the pushforward of $dm$ under the natural projection.

\begin{proposition}\label{defect} Let $\{Y_i=\TT_{\psi_i}g_i\}$ be a sequence of homogeneous toral sets with $\psi_i\colon K_i\to B$ such that $K_i\notin\mathcal{K}(M)$. Assume also that its sequence of discriminants approaches $\infty$ as $i\to\infty$. Then the projection of $Y_i$ to $[\GG]_M$ becomes equidistributed with respect to $dm_M$. In other words, the sequence $\{\mu_{{Y_i},M}\}$ converges weak* to $dm_M$.
\end{proposition}

\begin{proof} In terms of functionals over $C([\GG]_M)$, the pushforward of $\mu_{Y_i}$ corresponds to the restriction of $\mu_{Y_i}$ over the space of $M$-right invariant functions over $[\GG]$. Therefore, by theorem \ref{metathm} and its following discussion, to prove equidistribution we need to check \[\int_{[\GG]}f(g)d\mu(g)=\int_{[\GG]}f(g)dm(g),\] with $\mu$ a weak* accumulation measure of $\mu_{Y_i}$ and $f$ of the form $\chi\circ\mathrm{nrd}$ with an $M$-right invariant character $\chi$. This last condition implies that $K_\chi\in\mathcal{K}(M)$ and now proposition \ref{test} allows us to conclude since after taking a subsequence we can assume that $\mu_{Y_i}$ weak* converges to $\mu$ and then $\int_{[G]}f(g)d\mu(g)=\lim\int_{Y_i}f(g)d\mu_{Y_i}(g)$.
\end{proof}

The map $H \to \Delta_\psi$ sending $h$ to $\Gamma gh$ induces a uniformization of $\Delta_\psi$ by $H/(H\cap g^{-1}\Gamma g)$. Using a Haar measure in $H$, we equip $H/(H\cap g^{-1}\Gamma g)$ with a finite measure and push it to a finite $H$-right invariant  measure $\nu_\psi$ in $\Delta_\psi$. We equip $\Omega_\O=\bigsqcup_{[\psi]\in opt(\O,R)}\Delta_\psi$ with the unique probability measure $\mu_\O$ proportional to $\sum_{opt(\O,R)}\nu_\psi$. There exits a unique probability measure on the quotient $\lquot{\Gamma}{G}$ induced by the counting measure on $\Gamma$ and a Haar measure on $G$ satisfying the property analogue to (\ref{unfold}). It is the pushforward of the measure $dm$ by means of the projection $[\GG]\to \lquot{\Gamma}{G}$ described in (\ref{ident}).

\begin{proposition}\label{prueba} The collection $\Omega_\O$ becomes equidistributed on $\lquot{\Gamma}{G}$ as $d\to-\infty$. In particular, the collection $Ih(\O)$ becomes equidistributed in $\lquot{\Gamma^+}{\H\times \PGL_2(\Q_p)}$ and $Heeg(\O)$ becomes equidistributed in $\lquot{\Gamma}{\H_p}$.
\end{proposition}

\begin{proof}
Remember the identification $[\GG]_{\KK}=\lquot{\Gamma}{G}$. Fix $\psi_0\in opt(\O,R)$ and let $Y_\O$ the homogeneous toral set $\TT_{\psi_0} g$ with $g$ as before. Since $\det(\widehat{R}^\times_\#)=\widehat{\Z}_\#^\times$, the set $\mathcal{K}(\KK_f)$ consist only of quadratic fields ramified at $p$, i.e $\Q(\sqrt{\pm p})$. But our sequence of tori consider $K$ split or inert at $p$, so we will always be in the case $K_i\neq K_\chi$ and therefore by Theorem \ref{metathm} together with Proposition \ref{defect} there's equidistribution of the image of $Y_\O$ in $[\GG]_{\KK}$ if we show that the discriminant goes to $\infty$. Note that Proposition \ref{projection} shows that $Y_\O$ projects onto $\Omega_\O$.

Take $R$ an $\Z$-order in $B(\Q)$ agreeing with $R$ outside $p$. We know $\psi(K)\cap R=\psi(\O_{d}[1/p])$. Localizing at $q\neq p$ we have $g_q=1$ which implies $\Lambda_q=\O_d\otimes\Z_q$. Then the discriminant of $Y_\O$ is the discriminant of $\O_{dp^m}$ for some $m\geq0$. Certainly this goes to infinity and we must have equidistribution.
\end{proof}

\begin{remark} In the indefinite case, since $g_p$ is chosen so that $\psi(K_p)^\times=g_p Ag_p^{-1}$ we have that $g_pv_0$ belongs to the geodesic fixed by $\psi(K_p)^\times$. This is equivalent to $gv_0$ having stabilizer by the maximal order of $\psi(K_p)^\times$ which in turn is equivalent to $\psi(K_p)^\times\cap g_pM_2(\Z_p)g_p^{-1}$ being maximal. This shows that the discriminant of $Y$ is exactly the discriminant of $\O_d$.
\end{remark}

\selectlanguage{english}
\bibliographystyle{alpha}
\bibliography{biblio}

\end{document}